\documentclass[journal]{IEEEtran}
\def\baselinestretch{0.945}
\ifCLASSINFOpdf
\else
\fi
 \usepackage[caption=false,font=footnotesize]{subfig}
\newtheorem{theorem}{Theorem}
\newtheorem{lemma}{Lemma}
\newtheorem{remark}{Remark}
\newtheorem{definition}{Definition}
\newtheorem{corollary}{Corollary}
\newtheorem{problem}{Problem}
\newtheorem{procedure}{Procedure}
\newtheorem{assumption}{Assumption}

\usepackage{comment}

\newcommand{\abs}[1]{\ensuremath{\left\vert#1\right\vert}}

\newcommand{\cfunof}[1]{\ensuremath{\left\{#1\right\}}}

\newcommand{\C}{\ensuremath{\mathbb{C}}}

\newcommand{\discalg}{\ensuremath{\mathscr{A}_0}}

\newcommand{\ESPR}{\ensuremath{\mathrm{ESPR}}}
\newcommand{\funof}[1]{\ensuremath{\left(#1\right)}}

\newcommand{\Hfty}{\ensuremath{\mathscr{H}_{\infty}}}

\newcommand{\jw}{\ensuremath{\left(j\omega\right)}}
\newcommand{\jwo}{\ensuremath{\left(j\omega_0\right)}}
\newcommand{\lft}[3]{\ensuremath{\mathcal{F}_{#1}\left({#2},{#3}\right)}}

\newcommand{\norm}[1]{\ensuremath{\left\Vert #1 \right\Vert}}

\newcommand{\PR}{\ensuremath{\mathrm{PR}}}

\newcommand{\RH}{\ensuremath{\mathscr{RH}_{\infty}}}

\newcommand{\R}{\ensuremath{\mathbb{R}}}
\newcommand{\Rat}{\ensuremath{\mathscr{R}}}

\newcommand{\s}{\ensuremath{\left(s\right)}}
\newcommand{\sqfunof}[1]{\ensuremath{\left[#1\right]}}

\newcommand{\tm}{\ensuremath{\left(t\right)}}

\newcommand{\wo}{\ensuremath{\left(\omega_0\right)}}

\usepackage{ifthen}
\newboolean{showcomments}
\setboolean{showcomments}{true}
\usepackage[usenames,dvipsnames,table]{xcolor}
\usepackage{todonotes}

\definecolor{bleudefrance}{rgb}{0.19, 0.55, 0.91}
\definecolor{ao(english)}{rgb}{0.0, 0.5, 0.0}
\definecolor{brickred}{HTML}{AA2F40}
\definecolor{dark-gray}{gray}{0.4}

\newcommand{\addcite}[0]{\ifthenelse{\boolean{showcomments}}
{\textcolor{purple}{(add cite(s)) }}{}}%
\newcommand{\addcites}[0]{\ifthenelse{\boolean{showcomments}}
{\textcolor{purple}{(add cite(s)) }}{}}%

\newcommand{\enrique}[1]{  \ifthenelse{\boolean{showcomments}}
{\todo[inline,color=bleudefrance]{Enrique: #1}}{}}
\newcommand{\emmargin}[1]{\ifthenelse{\boolean{showcomments}}{\marginpar{\color{bleudefrance}\scriptsize EM: #1}}{}}

\newcommand{\richard}[1]{  \ifthenelse{\boolean{showcomments}}
{\todo[inline,color=Maroon]{Richard: #1}}{}}
\newcommand{\rpmargin}[1]{\ifthenelse{\boolean{showcomments}}{\marginpar{\color{Maroon}\scriptsize RP: #1}}{}}

\newboolean{showedits}
\setboolean{showedits}{false}
\usepackage[markup=underlined]{changes}
\definechangesauthor[color=bleudefrance]{EM}
\newcommand{\aem}[1]{
\ifthenelse{\boolean{showedits}}
{\added[id=EM]{#1}}
{\!#1\hspace{-4.75pt}}
}
\newcommand{\repem}[2]{
\ifthenelse{\boolean{showedits}}
{\replaced[id=EM]{#1}{#2}}
{\!#1\hspace{-4.75pt}}
}
\newcommand{\dem}[1]{
\ifthenelse{\boolean{showedits}}
{\deleted[id=EM]{#1}}
{}
}

\makeatletter
\g@addto@macro\normalsize{%
 \setlength\abovedisplayskip{5pt plus 5pt minus .5pt} 
 \setlength\belowdisplayskip{5pt plus 5pt minus .5pt} 
}
\makeatother

 \makeatletter
 \let\origsection\section
 \renewcommand\section{\@ifstar{\starsection}{\nostarsection}}
 \newcommand\nostarsection[1]
 {\sectionprelude\origsection{#1}\sectionpostlude}
 \newcommand\starsection[1]
 {\sectionprelude\origsection*{#1}\sectionpostlude}

 \newcommand\sectionprelude{%
   \vspace{-.5pt}
 }
 \newcommand\sectionpostlude{%
   \vspace{-.5pt}
 }

\usepackage{siunitx}
\usepackage{amsmath}
\usepackage{amsfonts}
\usepackage{amssymb}
\usepackage{mathtools}
\usepackage{nicefrac} 
\usepackage{mathrsfs}
\usepackage{microtype}
\usepackage[dvipsnames]{xcolor}
\usepackage{graphicx}
\usepackage{tikz}
\usepackage{circuitikz}
\usetikzlibrary{arrows,decorations.pathreplacing}

\def\c{white}
\tikzset{every path/.style={thick}}
\tikzset{>=stealth}
\usepackage{pgfplots}
\pgfplotsset{compat=1.13}

\usepackage{cases}
\usepackage{cleveref,graphicx,overpic}
\usepackage[toc,acronym,nonumberlist]{glossaries}
\usepackage{enumerate}
\usepackage[dvipsnames]{xcolor}

\newenvironment{proof}{\IEEEproof}{\endIEEEproof}

\usepackage{cleveref}
\usepackage{autonum}
\crefname{problem}{problem}{problems}
\crefname{proposition}{proposition}{propositions}
\crefname{procedure}{procedure}{procedures}
\crefname{assumption}{assumption}{assumptions}
\newacronym{agc}{AGC}{Automatic Generation Control}
\newacronym{ace}{ACE}{Area Control Error}
\newacronym{pr}{PR}{Positive Real}
\newacronym{spr}{SPR}{Strictly Positive Real}
\newacronym{espr}{ESPR}{Extended Strictly Positive Real}
\newacronym{kyp}{KYP}{Kalman-Yakobovich-Popov}
\newacronym{lky}{LKY}{Lefschetz-Kalman-Yakubovich}
\newacronym{lmi}{LMI}{Linear Matrix Inequality}
\newacronym{lft}{LFT}{Linear Fractional Transformation}

\newcommand*\diag[0]{\mbox{diag}}

\usepackage{booktabs} 

\newcommand\oprocendsymbol{\hbox{$\square$}}
\newcommand\oprocend{\relax\ifmmode\else\unskip\hfill\fi\oprocendsymbol}

\begin{document}
%
\title{Robust Scale-Free Synthesis for Frequency Control in Power Systems}
%
%
%

\author{Richard~Pates and Enrique~Mallada

\thanks{R. Pates is a member of the LCCC Linnaeus Center and the ELLIIT Excellence Center at Lund University, Lund, Sweeden. Email: \text{richard.pates@control.lth.se}. 
E. Mallada is with the Department of ECE at Johns Hopkins University, Baltimore, Maryland, USA. Email:\text{mallada@jhu.edu}.
This work was supported by the Swedish Foundation for Strategic Research, the Swedish Research Council through the LCCC Linnaeus Center, and NSF through grants CNS 1544771, EPCN 1711188, AMPS 1736448, and CAREER 1752362. A preliminary version of this work has been presented in \cite{PM17}.}

}

\maketitle

\begin{abstract}
The AC frequency in electrical power systems is conventionally regulated by synchronous machines. The gradual replacement of these machines by asynchronous renewable-based generation, which provides little or no frequency control, increases system uncertainty and the risk of instability. This imposes hard limits on the proportion of renewables that can be integrated into the system. In this paper we address this issue by developing a framework for performing frequency control in power systems with arbitrary mixes of conventional and renewable generation. Our approach is based on a robust stability criterion that can be used to guarantee the stability of a full power system model on the basis of a set of decentralised tests, one for each component in the system. It can be applied even when using detailed heterogeneous component models, and can be verified using several standard frequency response, state-space, and circuit theoretic analysis tools. Furthermore the stability guarantees hold independently of the operating point, and remain valid even as components are added to and removed from the grid. By designing decentralised controllers for individual components to meet these decentralised tests, every component can contribute to the regulation of the system frequency in a simple and provable manner. Notably, our framework certifies the stability of several existing (non-passive) power system control schemes and models, and allows for the study of robustness with respect to delays.
\end{abstract}


\IEEEpeerreviewmaketitle

\section{Introduction}


The composition of the electric grid is in a state of flux~\cite{Milligan:2015ju}. 
Motivated by the need to reduce carbon emissions, conventional synchronous generators, with relatively large inertia, are being replaced with renewable energy sources with little (wind) or no inertia (solar) at all~\cite{Winter:2015dy}. 
In addition, the steady increase of power electronics on the demand side is gradually diminishing the load sensitivity to frequency variations~\cite{WoodWollenberg1996}.
As a result, rapid frequency fluctuations are becoming a major source of concern for several grid operators~\cite{Boemer:2010wa,Kirby:2005uy}. Besides increasing the risk of frequency instabilities, this dynamic degradation also places limits on the total amount of renewable generation that can be sustained by today's electric grids. Ireland, for instance, is already resorting to wind curtailment whenever wind production exceeds $50\%$ of existing demand in order to preserve the grid stability.

One approach that has been proposed to mitigate this degradation is to use inverter-based generation to mimic synchronous generator behaviour, by implementing so called virtual inertia~\cite{Driesen:ft}. The rationale  is that by mimicking synchronous generator dynamics, virtual inertia will restore the robust frequency regulation that the system used to enjoy. However, it is unclear whether this particular choice of control is the most suitable for the task. 
Unlike generator dynamics that set the grid frequency, virtual inertia controllers estimate the grid frequency and its derivative using noisy and delayed measurements, which can lead to noise amplification and instabilities~\cite{m2016cdc,jpm2017cdc}.
Furthermore, inverter-based control can be significantly faster than that available for conventional generators. Therefore, using inverters to mimic generator behaviour does not take advantage of their full potential. This poses a new challenge for the control system engineer: develop control systems to regulate frequency in power systems that exploit the capabilities of inverters, and that overcome the issues introduced by renewable generation, including uncertainty in supply, measurement delays, network topology changes, and heterogeneity among components.


To achieve this goal, new methods for controller synthesis are required. The crux of the issue is that in the power system context, in order to ensure secure operation, control systems must be able to guarantee \textbf{in advance} that adequate levels of robustness are maintained even if its operating point changes, and as components join and leave the grid. Given their uncertain nature, increasing the number of renewable sources vastly increases the number of ways this can happen. It then becomes very difficult to apply conventional control design methods, since one cannot determine which model to use, or identify a tractable set of operating points or network configurations to consider. This is an issue even for many specialised methods for large systems, such as those based on small gain or dissipativity theory \cite{BL+06,arcak2016networks}. This is because these still typically require the verification of the feasibility of a \gls{lmi} that scales with the size of the network, and this test would have to be rechecked for every operating point and change in network configuration.


In this paper, we argue that the best way to address the challenge of achieving robustness and scalability is `to get the local design right'. To do so, we look to follow, and further extend, the philosophy of passivity based design, and find conditions on the subsystems in the network that guarantee robust stability \textbf{independently} of how they are interconnected. These conditions can then be used as a principled basis for \textbf{scale-free} design that addresses the requirements of the network setting. In particular, by designing controllers to meet a local stability requirement, strong a-priori guarantees --that hold even as the operating point changes, and as components are added to or removed from the network-- can be given.

Our main contribution, presented as \Cref{thm:main} in \Cref{sec:res1}, is to derive a decentralised stability criterion that is tailored to frequency control problems in power systems. As described in \Cref{sec:appthm}, the condition allows stability of a full power system model to be deduced on the basis of a set tests on the individual components in the network, in a manner that is independent of operating point and interconnection configuration. The condition allows for detailed, heterogeneous components models, and can include the effect of delays. As shown in \Cref{sec:an}, the criterion is robust, and can be verified using several standard frequency response, state-space, and circuit theory analysis tools. Furthermore, as discussed in \Cref{sec:res3}, it allows for the synthesis of controllers using only local models. The design can be conducted using standard frequency response intuition, as well with off-the-shelf tools from $\Hfty{}$ optimal control. As explained in \Cref{sec:d} standard passivity based design criteria arise as a special case, and there essentially exist no better criteria that can be used as a basis for decentralised design with a-priori stability guarantees. We illustrate the results on several standard power system models and controller architectures in \Cref{sec:examples}.

\subsubsection*{Notation}

$\Hfty$ denotes the space of transfer functions of stable linear, time-invariant systems. This is the Hardy space of functions that are analytic on the open right half plane $\C_+$ with bounded norm $\norm{g\s}_\infty\coloneqq{}\sup_{s\in\C_+}\abs{g\s}$. $\discalg$ denotes the subset of $\Hfty$ that is continuous on the extended imaginary axis \cite{Par97}. $\Rat$ denotes the set of real rational functions, and $\RH\coloneqq{}\Rat\cap\Hfty$. Finally, we denote the lower \gls{lft} as $\lft{l}{G}{C}\coloneqq{}G_{11}+G_{12}C\funof{I-G_{22}C}^{-1}G_{21}$.

\section{Problem Description}

\begin{figure}[tp]
\centering
\begin{tikzpicture}[]font=\footnotesize]
  
  \node[draw, circle, inner sep=0, radius=0.2] (p) at (0,0) {$+$};
  \draw[->] (-1,0) -- node[above] {$P_d$} (p.west);
  \draw[->] (p.east) -- (1,0);
  
  \node[align=center] at (2,1.35) {Bus Dynamics
  };
  
  \draw (1,-1) rectangle (3,1);
  \draw[fill=\c] (1,1) rectangle node[midway] {$g_1$} (1.4,0.6);
  \draw[fill=\c] (1.4,0.6) rectangle (1.8,0.2);
  \draw[fill=\c] (1.8,0.2) rectangle node[midway] {$g_i$} (2.2,-0.2);
  \draw[fill=\c] (2.2,-0.2) rectangle (2.6,-0.6);
  \draw[fill=\c] (2.6,-0.6) rectangle node[midway] {$g_n$} (3,-1);
  
  \draw[->] (3,0) -- node[near end, above] {$\dot{\theta}$} (5,0);
  \draw[->] (4,0) -- (4,-2) -- (2.5,-2);
  
  \draw[fill=\c] (1.5,-2.4) rectangle node[midway] {$\cfrac{1}{s} \, L_B$} (2.5,-1.6);
  \node[align=center] at (2,-2.75) {Network};
  
  \draw[->] (1.5,-2) -- (0,-2) node[near end, below] {$P_N$} -- (p.south) node[left, yshift=-4] {$\text{--}$};
  
\end{tikzpicture}
\caption{Block diagram of the linearised power system model, where $g_i\s=\lft{l}{G_i\s}{c_i\s}$ as illustrated in \cref{fig.BusDyn}.}\label{fig.GL}
\end{figure}
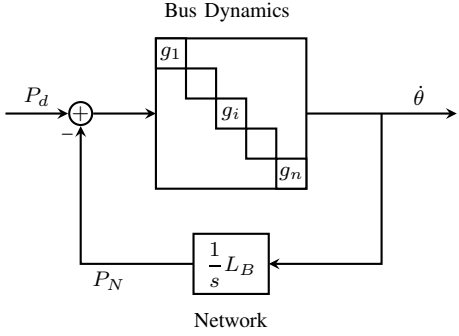

In this section we describe the power system model used in this paper. We model the power system as a set of $n$ buses, indexed by $i\in\{1,\dots,n\}$, which are coupled through an AC network. Assuming operation around an equilibrium, the linearised dynamics are represented by the block diagram in \Cref{fig.GL}. The transfer function $g_i\s$ describes the dynamics of the components connected at the \emph{i}th bus. The input to each $g_i\s$ is the net power flow into the bus, relative to its equilibrium value. This includes the variation $P_{N,i}$ in electrical power drawn from the network and an external disturbance $P_{d,i}$, which reflects, for example, variations in power drawn by local loads. The output of each $g_i\s$ is the rate of change of voltage angle (frequency) at the given bus.

The network power fluctuations $P_N$ are given by a linearised DC model of the power flow equations. More precisely, 
\begin{equation}
    \label{eq:network}
    P_N\s = \frac{1}{s}L_B\dot\theta\s
\end{equation} where $L_B$ is an undirected weighted Laplacian matrix with entries given by 
\begin{equation}\label{eq:Lap}
L_{B,ij}=\frac{\partial}{\partial{}\theta_j}{\sum_{l=1}^nV_{i0}V_{l0}b_{il}\sin\funof{\theta_i-\theta_l}}\Bigr|_{\theta=\theta_0}.
\end{equation}
In the above $V_{0}\in\R^n$ and $\theta_0\in\R^n$ denote the voltage magnitudes and angles at the buses in steady state, and $b_{il}\geq{}0$ the susceptance of the transmission line connecting buses \emph{i} and \emph{l} ($b_{il}=0$ if there is no line). 

Finally, to allow for the design of local controllers, we further open the loop at each $g_i\s$ and define a generalized plant model $G_i\s$ for each bus as
\begin{equation}\label{eq:generalized-plant}
        \begin{bmatrix}
        \dot\theta_i\s\\z_i\s
        \end{bmatrix}=
        \begin{bmatrix}
        G_{i,11}\s \!\!&\! G_{i,12}\s\\
        G_{i,21}\s \!\!&\! G_{i,22}\s
        \end{bmatrix}\!\!
        \begin{bmatrix}
        P_{d,i}\s-P_{N,i}\s\\P_{c,i}\s
        \end{bmatrix}\!.
\end{equation}
The entries of $G_i\s$ capture both the internal dynamics at the bus, and specify the measurements available for control system design. The signal $z_i\s$ specifies the measurements available for implementing the local controller, and $P_{c,i}\s$ the controller's power injection. These signals are related through
\begin{equation}\label{eq:controller}
    P_{c,i}\s=c_i\s z_i\s,
\end{equation} 
where $c_i\s$ is the transfer function of the controller to be designed. The transfer functions $g_i,G_i$ and $c_i$ are related through the lower \gls{lft} according to $g_i=\lft{l}{G_i}{c_i}$ as illustrated in \Cref{fig.BusDyn}. Note that in general $G_i$ and $c_i$ need not be scalar, though $g_i$ always is. Combining \cref{eq:network,eq:generalized-plant,eq:controller} leads to the following
generic linearised power system model:
\begin{equation}\label{eq:model}
\begin{aligned}
\begin{bmatrix}
\dot\theta_i\s\\z_i\s
\end{bmatrix}&=
G_i\s
\begin{bmatrix}
P_{d,i}\s-P_{N,i}\s\\P_{c,i}\s
\end{bmatrix},\\
P_{c,i}\s&=c_i\s{}z_i\s,\\
P_N\s&=\frac{1}{s}L_B\dot{\theta}\s.
\end{aligned}
\end{equation}

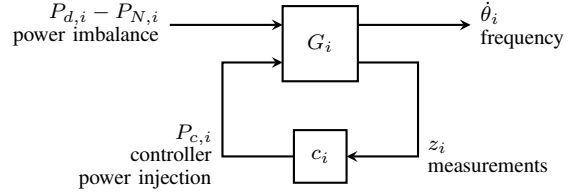
\begin{figure}[tp]
\centering
\begin{tikzpicture}[font=\footnotesize]
  
  \node[align=center] at (0.5,1.25) {};
  \node[align=center] at (0.5,-1.75) {};
  
  \draw (0,0) rectangle node[midway] {$G_i$} (1,1);
  \draw (0.15,-1.35) rectangle node[midway] {$c_i$} (.85,-0.65);
  
  \draw[->] (1,0.75) -- (2.5,0.75) node[right, align=left] {$\dot{\theta}_i$\\frequency};
  \draw[->] (1,0.25) -- (1.8,0.25) -- (1.8,-1) node[right, align=left] {$z_i$\\measurements} -- (.85,-1);
  \draw[->] (0.15,-1) -- (-.8,-1) node[left, align=right] {$P_{c,i}$  \\controller\\power injection} -- (-.8,0.25) -- (0,0.25);
  \draw[->] (-1.5,0.75) node[left, align=right] {$P_{d,i}-P_{N,i}$\\power imbalance} -- (0,0.75);
  
\end{tikzpicture}
\caption{Generalized plant description of the dynamics at the \emph{i}th bus. The transfer function from the power imbalance to frequency is $g_i=\lft{l}{G_i}{c_i}$.}\label{fig.BusDyn}
\end{figure}

Although \cref{eq:model} is rather generic and can account for many bus models, when illustrating our approach we will use models based on the classical swing equations. That is, we will consider the bus dynamics described by 
\[
m_i \ddot{\theta_i}+d_i\dot \theta_i = P_{c,i}+P_{d,i}-P_{N,i},
\]
where $m_i$ and $d_i$ are the generator's inertia and damping respectively. This leads to a generalised plant transfer function
\[
G_i\s=\begin{bmatrix}
\frac{1}{m_is+d_i}&\frac{1}{m_is+d_i}\\
G_{i,21}\s \!\!&\! G_{i,22}\s
\end{bmatrix},
\]
where the particular transfer functions $G_{i,21}\s$ and $G_{i,22}\s$ depend on the measured signal $z_i\s$. For example, if angular velocity measurements are available, then $G_{i,21}\s=G_{i,22}\s=\frac{1}{m_is+d_i}$.  

\begin{remark}
The network model in \cref{eq:model} implicitly makes the following assumptions which are standard and well-justified for frequency control in transmission networks \cite{kundur_power_1994}: (i) bus voltage magnitudes are constant for all $i$,
(ii) transmission lines are lossless, and (iii) reactive power flows do not affect bus voltage phase angles and frequencies. See, e.g., \cite{Zhao:2014bp,Li:2016tcns,mallada2017optimal} for applications of similar models for frequency control within the control literature.
\end{remark}

\section{Results}\label{sec:res}

\subsection{A Scale-Free Stability Criterion}\label{sec:res1}

%
%
%
%
%
%

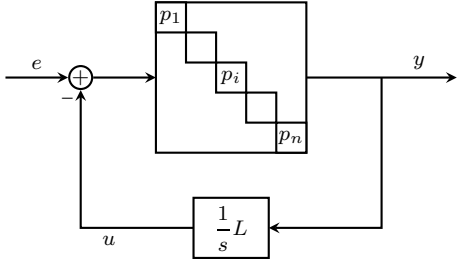
\begin{figure}
\centering
\begin{tikzpicture}[]font=\footnotesize]
  
  \node[draw, circle, inner sep=0, radius=0.2] (p) at (0,0) {$+$};
  \draw[->] (-1,0) -- node[above] {$e$} (p.west);
  \draw[->] (p.east) -- (1,0);
  
  
  \draw (1,-1) rectangle (3,1);
  \draw[fill=\c] (1,1) rectangle node[midway] {$p_1$} (1.4,0.6);
  \draw[fill=\c] (1.4,0.6) rectangle (1.8,0.2);
  \draw[fill=\c] (1.8,0.2) rectangle node[midway] {$p_i$} (2.2,-0.2);
  \draw[fill=\c] (2.2,-0.2) rectangle (2.6,-0.6);
  \draw[fill=\c] (2.6,-0.6) rectangle node[midway] {$p_n$} (3,-1);
  
  \draw[->] (3,0) -- node[near end, above] {$y$} (5,0);
  \draw[->] (4,0) -- (4,-2) -- (2.5,-2);
  
  \draw[fill=\c] (1.5,-2.4) rectangle node[midway] {$\cfrac{1}{s} \, L$} (2.5,-1.6);
  
  \draw[->] (1.5,-2) -- (0,-2) node[near end, below] {$u$} -- (p.south) node[left, yshift=-4] {$\text{--}$};
  
\end{tikzpicture}
\caption{\Cref{thm:main} shows that given any $h\in\PR\cap\discalg{}$, stability of this feedback interconnection is guaranteed for all $p_i\in\mathcal{P}_h$ and all $L\in\mathcal{L}$.}\label{fig.GLP}
\end{figure}

In this section we will present a scale-free stability criterion for the feedback interconnection
\begin{equation}\label{eq:seceq}
    \begin{aligned}
    y_i\s&=p_i\s{}\funof{e_i\s-u_i\s}\\
    u\s&=\frac{1}{s}Ly\s.
    \end{aligned}
\end{equation}
This interconnection is illustrated in \Cref{fig.GLP}. In particular we will show that given any $L$ in the set
\[
\mathcal{L}\coloneqq{}\cfunof{L:L=L^T,0\preceq{}L\preceq{}I},
\]
stability\footnote{We say the interconnection is stable if
\[
\begin{bmatrix}
P\s\\I
\end{bmatrix}\funof{I+\tfrac{1}{s}LP\s{}}^{-1}\begin{bmatrix}
\tfrac{1}{s}L&I
\end{bmatrix}\in\Hfty{}^{2n\times{}2n},
\]
where $P\s=\diag\funof{p_1\s,\ldots{}p_n\s}$.
} of \cref{eq:seceq} can be guaranteed on the basis of decentralised tests on each of the transfer functions $p_i\s$. We will show how to use this to guarantee stability of the linearised power system model in the next section.

Our criterion is written in terms of \gls{pr} and \gls{espr} functions. This establishes strong connections to many well established areas of control theory, including:
\begin{enumerate}
    \item Multiplier methods and absolute stability criteria;
    \item $\Hfty$ optimal control;
    \item The Nyquist stability criterion;
    \item Classical circuit theory.
\end{enumerate}
We will highlight these connections throughout the rest of the paper. We now formally define these function classes.

\begin{definition}\label{def:pr}
A  (not necessarily proper or rational) transfer functions $g\s$ is \gls{pr} if:
\begin{enumerate}[(i)]
\item $g\s$ is analytic in $\mathrm{Re}\funof{s}>0$;
\item $g\s$ is real for all positive real $s$;
\item $\mathrm{Re}\funof{g\s}\geq{}0$ for all $\mathrm{Re}\funof{s}>0$.
\end{enumerate}
If in addition $g\in\discalg$ and there exists an $\epsilon>0$ such that $g\funof{s}-\epsilon$ is \gls{pr}, then $g\s$ is \gls{espr}.
\end{definition}

The following theorem, which is inspired by the results for scalar systems from \cite[Theorem 2]{BW65}, shows that provided $L\in\mathcal{L}$ and that the elements in the diagonal transfer function are drawn from a parametrised class
\[
\mathcal{P}_h\coloneqq{}\cfunof{p\in\Hfty:p\funof{0}\neq{}0,h\s\funof{1+\frac{p\s}{s}}\in\ESPR},
\]
then the feedback interconnection in \cref{eq:seceq} is stable.

\begin{theorem}\label{thm:main}
If $h\in\PR\cap\discalg{}$, then for any $p_1,\ldots{},p_n\in\mathcal{P}_h$ and any $L\in\mathcal{L}$, the feedback interconnection in \cref{eq:seceq} is stable.
\end{theorem}

\begin{remark}
The function $h\s$ in \Cref{thm:main} is typically referred to as a multiplier. A useful class of multipliers that we will use in all our examples is given by
\[
\cfunof{\frac{s}{s+T}\prod_{k=1}^N\frac{s+\alpha_k}{s+\beta_k}:0<\beta_1<\alpha_1<\beta_2<\ldots{}<T}.
\]
There is an extensive literature supporting the design of multipliers \cite{BW65}, and (as we will discuss in \Cref{sec:fr}) the choice of $h\s$ has a graphical interpretation. Nonlinear extensions of \Cref{thm:main} are also possible, using for example the Popov or Zames-Falb multipliers, though this will not be pursued here (see \cite{PV17} for ideas along these lines).
\end{remark}

\begin{proof}
Let $P=\diag\funof{p_1,\ldots{},p_n}$. Since $P\in\Hfty^{n\times{}n}$, the interconnection of $P$ and $\frac{1}{s}L$ is stable if and only if
\[
\tfrac{1}{s}L\funof{I+\tfrac{1}{s}PL}^{-1}\in\Hfty^{n\times{}n}.
\]
Since $L\in\mathcal{L}$, we can factorize it as $L=QXQ^*$, where $\epsilon{}I\preceq{}X\preceq{}I$, $Q\in\C^{n\times{}\funof{n-m}}$, $m>0$, $Q^*Q=I$, $\epsilon>0$. 
Hence 
\[
\begin{aligned}
\tfrac{1}{s}L\funof{I+\tfrac{1}{s}PL}^{-1}&=QXQ^*\funof{sI+PQXQ^*}^{-1},\\
&=QX\funof{sI+Q^*PQX}^{-1}Q^*.
\end{aligned}
\]
Clearly then it is sufficient to show that
\begin{equation}\label{eq:intintint1}
\funof{sI+Q^*PQX}^{-1}\in\Hfty^{\funof{n-m}\times{}\funof{n-m}}.
\end{equation}
The above can be immediately recognised as an eigenvalue condition: $-s\notin\lambda\funof{Q^*P\s{}QX},\forall{}s\in\overline{\mathbb{C}}_+$. By Theorem 1.7.6 of \cite{HJ91}, for any $s\in\C$:
\[
\lambda\funof{Q^*P\s{}QX}\subset\text{Co}\funof{kp_i\s:i\in \cfunof{1,\ldots{},n},\epsilon\leq{}k\leq{}1}.
\]
Therefore it is sufficient to show that
\begin{equation}\label{eq:convsets}
0\notin\text{Co}\funof{s+kp_i\s:i\in \cfunof{1,\ldots{},n},\epsilon\leq{}k\leq{}1},
\end{equation}
for all $s=\overline{\mathbb{C}}_+$. Observe that since each $p_i\s$ is bounded, this condition is trivially satisfied for large $s$. It is therefore enough to check that this holds for $s\in\overline{\mathbb{C}}_+,\abs{s}<R,$ for sufficiently large $R$. This can be done using the separating hyperplane theorem, applied pointwise in $s$. In particular, \cref{eq:convsets} holds for any given $s$ if and only if there exists a nonzero $\alpha\in\C$ and $\gamma>0$ such that $\forall{}i\in\cfunof{i,\ldots{},n}$:
\begin{equation}\label{eq:prtest}
\text{Re}\funof
{\alpha{}\funof{s+kp_i\s}}\geq{}\gamma{},\forall{}\;\epsilon{}\leq{}k\leq{}1.
\end{equation}
We will now use a minor adaptation of the argument in Theorem 2 of \cite{BW65} to show that such an $\alpha{}$ is guaranteed to exist. From the conditions of the theorem and the maximum modulus principle, for any $R\geq{}0$, there exists a $\delta>0$ such that $\forall{s}\in\overline{\C}_+,\abs{s}\leq{}R$: 
\begin{equation}\label{eq:finaltest}
\text{Re}\funof{h\s\funof{1+p_i\s/s}}\geq{}\delta.
\end{equation}
Since $h\s$ is \gls{pr}, for all $k^*\geq{}0$, $\text{Re}\funof{k^*h\s}\geq{}0$, and therefore
\[
\text{Re}\funof{h\s\funof{1+p_i\s/s}+k^*h\s}\geq\delta{}.
\]
Dividing through by $\funof{1+k^*}$ and rearranging shows that under these conditions
\[
\text{Re}\funof{\tfrac{h\s}{s}\funof{s+\funof{1+k^*}^{-1}p_i\s}}\geq{}\funof{1+k^*}^{-1}\delta{}.
\]
Therefore setting $\alpha\equiv{}h\s/s$ and $\gamma\equiv\epsilon{}\delta{}$ shows that \cref{eq:prtest} is satisfied for the required values of $k$ and $s$. Consequently \cref{eq:intintint1} is satisfied, and the result follows.
\end{proof}

\subsection{Applying \Cref{thm:main} to Linearised Power System Models}\label{sec:appthm}

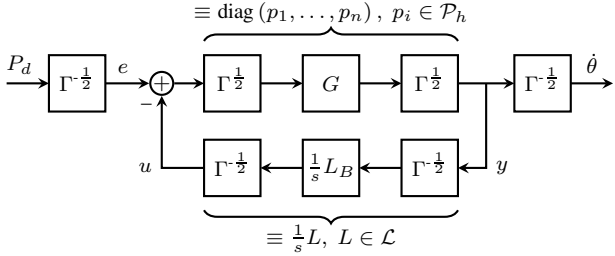
\begin{figure}[tp]
\centering
\begin{tikzpicture}[scale=0.75,font=\footnotesize]
  
  \draw (-0.25,0) rectangle node[midway,yshift=1.5] {$\Gamma^{\text{-}\frac{1}{2}}$} (0.75,1);
  \node[draw, circle, inner sep=0, radius=0.2] (p) at (1.75,0.5) {$+$};
  \draw (2.5,0) rectangle node[midway,yshift=1.5] {$\Gamma^{\frac{1}{2}}$} (3.5,1);
  \draw (4.25,0) rectangle node[midway] {$G$} (5.25,1);
  \draw (6,0) rectangle node[midway,yshift=1.5] {$\Gamma^{\frac{1}{2}}$} (7,1);
  \draw (8,0) rectangle node[midway,yshift=1.5] {$\Gamma^{\text{-}\frac{1}{2}}$} (9,1);
  
  \draw (2.5,-1.5) rectangle node[midway,yshift=1.5] {$\Gamma^{\text{-}\frac{1}{2}}$} (3.5,-0.5);
  \draw (4.25,-1.5) rectangle node[midway] {$\frac{1}{s}L_B$} (5.25,-0.5);
  \draw (6,-1.5) rectangle node[midway,yshift=1.5] {$\Gamma^{\text{-}\frac{1}{2}}$} (7,-0.5);
  
  \draw[->] (-1,0.5) -- node[above, pos=0.3] {$P_d$} (-0.25,0.5);
  \draw[->] (0.75,0.5) -- node[pos=0.4,above] {$e$} (p.west);
  \draw[->] (p.east) -- (2.5,0.5);
  \draw[->] (3.5,0.5) -- (4.25,0.5);
  \draw[->] (5.25,0.5) -- (6,0.5);
  \draw[->] (7,0.5) -- (8,0.5);
  \draw[->] (9,0.5) -- node[above, pos=0.5] {$\dot{\theta}$} (9.75,0.5);
  
  \draw[->] (7.5,0.5) -- (7.5,-1) node[right] {$y$} -- (7,-1);
  \draw[<-] (5.25,-1) -- (6,-1);
  \draw[<-] (3.5,-1) -- (4.25,-1);
  \draw[<-] (p.south) node[left, yshift=-4] {$\text{--}$} -- (1.75,-1) node[left] {$u$} -- (2.5,-1);
  
  \draw[decorate, decoration={brace, amplitude=5pt, mirror}] 
      (2.5,-1.75) -- (7,-1.75);
  \draw[decorate, decoration={brace, amplitude=5pt}] 
      (2.5,1.25) -- (7,1.25);
      
  \node (pu) at (4.75,1.75) {$\equiv\text{diag}\funof{p_1,\ldots{},p_n},\;p_i\in\mathcal{P}_h$};
  \node (pb) at (4.75,-2.25) {$\equiv\frac{1}{s}L,\;L\in\mathcal{L}$};
  
\end{tikzpicture}
\caption{Loop transformation used to apply \Cref{thm:main} to the power system model in \cref{eq:model}, where $G=\text{diag}\funof{g_1,\ldots{},g_n}$ (c.f. \Cref{fig.GL,fig.GLP}).}\label{fig:LoopTF}
\end{figure}

In this section we will show that a set of decentralised conditions can be used to guarantee stability of the full linearised power system model in \cref{eq:model}. These guarantees are valid for every operating point that satisfies the following mild assumption.
\begin{assumption}\label{ass:1}
At equilibrium, the angle difference $\abs{\theta_{0,i}-\theta_{0,j}}$ across each transmission line is less than \ang{90}, and the voltage magnitude at each bus is at most $V_{\max,i}$.
\end{assumption}
This assumption is essentially without loss of generality, since thermal and voltage drop limitations for transmission lines preclude load angles anywhere near \ang{90} and equilibrium bus voltages above 1.05 p.u.  \cite{kundur_power_1994}.

We will now show that given any $h\in\PR\cap\discalg{}$, the power system model in \cref{eq:model} is guaranteed to be stable if every bus model satisfies
\begin{equation}\label{eq:basic}
\gamma_i\lft{l}{G_i}{c_i}\in\mathcal{P}_h,
\end{equation}
where
\begin{equation}\label{eq:gam1}
\gamma_i\coloneqq{}2\sum_{j=1}^nV_{\max,i}V_{\max,j}b_{ij}.
\end{equation}
Note that $\gamma_i$ is a constant that depends only on the susceptances of the transmission lines connected to the \emph{i}th bus and the largest allowable voltage magnitudes at their endpoints. Therefore this condition is local, independent of the operating point, and guarantees stability even as the components are connected and disconnected from the buses. This makes \cref{eq:basic} an ideal basis for conducting scale-free design. 

In order to verify stability of the power system model using \Cref{thm:main}, we need to connect \cref{eq:model,eq:seceq}. As can be seen from \Cref{fig.GL,fig.BusDyn}, by closing all the local control loops the interconnection in \cref{eq:model} simplifies to
\begin{equation}\label{eq:contclosed}
    \begin{aligned}
    \dot{\theta}_i\s&=\lft{l}{G_i\s}{c_i\s}\funof{P_{d,i}\s-P_{N,i}\s}\\
    P_N\s&=\frac{1}{s}L_B\dot{\theta}\s.
    \end{aligned}
\end{equation}
This feedback configuration has the same form as \cref{eq:seceq} (compare \Cref{fig.GL,fig.GLP}), however \Cref{thm:main} cannot yet be applied since $L_B$ is not necessarily in $\mathcal{L}$. The following simple lemma, which is proved in \Cref{app:lem1}, shows that we can rescale \cref{eq:contclosed} so that it is of the appropriate form.
\begin{lemma}\label{lem:rescale}
Suppose that $L_B$ as given by \cref{eq:Lap} satisfies \Cref{ass:1}, and let $\Gamma=\diag\funof{\gamma_1,\ldots{},\gamma_n}$, where the $\gamma_i$'s are given by \cref{eq:gam1}. Then given any conformal partitioning of $\Gamma$ and $L_B$ such that
\[
\Gamma=\begin{bmatrix}
\Gamma_1&0\\0&\Gamma_2
\end{bmatrix},\,L_B=\begin{bmatrix}
L_{B,11}&L_{B,11}\\
L_{B,21}&L_{B,22}
\end{bmatrix},
\]
$0\preceq{}\Gamma_1^{-\frac{1}{2}}\funof{L_{B,11}-L_{B,12}L_{B,22}^{-1}L_{B,21}}\Gamma_1^{-\frac{1}{2}}\preceq{}I$.
\end{lemma}

The most basic consequence of \Cref{lem:rescale} is that given any operating point satisfying \Cref{ass:1},
\[
\Gamma^{-\frac{1}{2}}L_B\Gamma^{-\frac{1}{2}}\in\mathcal{L}.
\]
This suggests that in order to rescale \cref{eq:contclosed} so that \Cref{thm:main} can be applied, we should use the loop transform in \Cref{fig:LoopTF}. This shows that stability of \cref{eq:contclosed} is equivalent to that of
\begin{equation}
    \begin{aligned}
    y_i\s&=\gamma_i{}\lft{l}{G_i\s}{c_i\s}\funof{e_i\s-u_i\s}\\
    u\s&=\frac{1}{s}\Gamma^{-\frac{1}{2}}L_B\Gamma^{-\frac{1}{2}}y\s.
    \end{aligned}
\end{equation}
In the above the signals $y,u,e$ are re-scaled versions of $\dot{\theta},P_N$ and $P_{d}$. \Cref{thm:main} can now be applied by setting
\[
p_i\s\equiv{}\gamma_i{}\lft{l}{G_i\s}{c_i\s}\,\text{and}\,L\equiv{}\Gamma^{-\frac{1}{2}}L_B\Gamma^{-\frac{1}{2}}.
\]
This proves that \cref{eq:basic} is sufficient for stability of \cref{eq:model} for every operating point meeting \Cref{ass:1}. Therefore all that remains is to show that these claims hold even as components are disconnected from the buses. Suppose for now that we disconnect the components at the \emph{\funof{n-m}--$n$}th buses. These buses are now `floating', and may be eliminated using Kron reduction in the usual way. If this is done we obtain the following `reduced' version of \cref{eq:contclosed}:
\[
    \begin{aligned}
    \dot{\theta}_i\s&=\lft{l}{G_i\s}{c_i\s}\funof{P_{d,i}\s-P_{N,i}\s}\\
    P_N\s&=\frac{1}{s}\funof{L_{B,11}-L_{B,12}L_{B,22}^{-1}L_{B,21}}\dot{\theta}\s,
    \end{aligned}
\]
where $L_{B,22}\in\R^{m\times{}m}$. \Cref{lem:rescale} shows that \textbf{exactly the same} loop transform will also re-scale the reduced model so that \Cref{thm:main} can be applied. Therefore satisfying \cref{eq:basic} also implies stability when these components are removed. By simply re-indexing the buses, the same argument can be used to show that \cref{eq:basic} also implies stability even as any combination of components are removed.

\begin{remark}
Stability as we have defined it implies that if the external signals (the disturbances $P_d$) are bounded and tend to zero, then the internal signals $P_N,\dot{\theta}$ will tend to zero. This does not necessarily mean that the `state variables' $\theta$ will tend to their equilibrium values $\theta_0$, since they do not appear explicitly in the internal signals. However, since
\[
P_N=L_B\funof{\theta-\theta_0},
\]
it is clear that if $\lim_{t\rightarrow{}\infty}P_N\tm=0$, then $\lim_{t\rightarrow{}\infty}\theta\tm-\theta_0\in\text{Ker}\funof{L_B}$. Therefore because $L_B$ is a weighted Laplacian matrix, satisfying \cref{eq:basic} ensures that the phases differences (and hence power flows) across the transmission lines  will return to their equilibrium values.
\end{remark}

\subsection{A Scale-Free Analysis Method}\label{sec:an}

\Cref{thm:main} shows that given a function $h\in\PR\cap\discalg$, stability can be guaranteed on a component by component basis using \cref{eq:basic}. The true strength of this result is that it can be used to design controllers based only on local models with a-priori guarantees that hold independently of operating point and network configuration. However before considering synthesis questions, it is first instructive to understand how to check \cref{eq:basic}. 

Rather than simply checking that \cref{eq:basic} holds, instead we propose to find the largest $\gamma$ such that $\gamma\lft{l}{G_i}{c_i}\in\mathcal{P}_h$. This is justified by the following lemma, and useful because it will give our criteria robustness guarantees. It will also provide a synthesis objective as discussed in \Cref{sec:res3}. The proof is given in \Cref{app:1}.

\begin{lemma}\label{lem:marg}
Let $h\in\PR$ and $p\in\mathcal{P}_h$. If $0<\gamma\leq{}1$, then $\gamma{}p\in\mathcal{P}_h$.
\end{lemma}

Based on the above, we define the following scale-free analysis problem.
\begin{problem}\label{prob:analysis}
Given $h,G_i,c_i$
\[
\begin{aligned}
\text{maximize} \quad{}&\gamma\\
\text{subject to}\quad&\gamma\lft{l}{G_i}{c_i}\in\mathcal{P}_h.
\end{aligned}
\]
\end{problem}
Denoting the solution to this problem as $\gamma_i^*$, it follows from \Cref{lem:marg} that if $\gamma_i\leq{}\gamma^*_i$, then $\gamma{}_i\lft{l}{G_i}{c_i}\in\mathcal{P}_h$ (i.e. \cref{eq:basic} is satisfied), and the difference $\gamma^*_i-\gamma_i$ gives a measure of robustness. We now summarise some techniques for solving \Cref{prob:analysis}. These both illustrate how to solve the problem, and also give insight into how the function $h\s$ should be selected.

\begin{remark}
Robustness with respect to other standard classes of uncertainty can also be guaranteed by adding more constraints to \Cref{prob:analysis}, see for example \cite{Jon01}.
\end{remark}

\subsubsection{Frequency response methods}\label{sec:fr}

Probably the simplest way to check that a function is \gls{espr} is to plot its frequency response. These methods are also the most insightful, since they give $h\s$ and \cref{eq:basic} a graphical interpretation. The required result is the following, and is proved in \Cref{app:2}.

\begin{lemma}\label{lem:fr}
Let $g\in\discalg$. Then $g\in\ESPR$ if and only if there exists an $\epsilon>0$ such that
\[
\text{Re}\funof{g\jw}\geq{}\epsilon,\;\forall\omega\in\R\cup\cfunof{\infty}.
\]
\end{lemma}

\begin{figure}
    \centering
%
%
\begin{tikzpicture}

\begin{axis}[%
width=4.5cm,
height=4.5cm,
scale only axis,
xmin=-2,
xmax=1,
xmajorgrids,
ymin=-2,
ymax=1,
ymajorgrids,
major grid style={black},
title style={font=\footnotesize},xlabel style={font={\color{blue}\bfseries}},ylabel style={font=\footnotesize},ticklabel style={font=\footnotesize}
]

\addplot[area legend,opacity=.5,fill=white!80!black,draw=none,forget plot]
table[row sep=crcr] {%
x	y\\
-2	-1\\
0	1\\
1	1\\
1	-2\\
-2	-2\\
};

\addplot [color=black,solid,forget plot]
  table[row sep=crcr]{%
-2.09910785557811	-69.9553927936708\\
-2.09890280702228	-63.0686777318254\\
-2.09865065895666	-56.8580051341389\\
-2.09834060721201	-51.2568028528711\\
-2.09795937670274	-46.2050319990416\\
-2.09749066014505	-41.6485434790145\\
-2.09691443123294	-37.538497700684\\
-2.09620610519584	-33.8308412420027\\
-2.09533551434335	-30.4858348897678\\
-2.0942656601136	-27.4676280133881\\
-2.09295119633287	-24.7438747432791\\
-2.09133659102718	-22.2853878826489\\
-2.08935390656358	-20.0658269004914\\
-2.08692013080259	-18.0614167381282\\
-2.08393398645202	-16.250694516965\\
-2.08027214376525	-14.6142815664685\\
-2.07578476598755	-13.1346785036927\\
-2.07029033183475	-11.796081393581\\
-2.06356971106864	-10.5842173064785\\
-2.05535952675514	-9.48619786793882\\
-2.04534493299118	-8.49038966510115\\
-2.03315208503202	-7.58630062761797\\
-2.01834079787063	-6.76448172492174\\
-2.00039819797912	-6.01644348843112\\
-1.97873458500137	-5.33458693217964\\
-1.95268323488731	-4.71214834081465\\
-1.92150646112458	-4.14315702858136\\
-1.88441081863587	-3.62240443910097\\
-1.84057471323861	-3.1454217546789\\
-1.78919158762479	-2.7084614787941\\
-1.72953090759393	-2.30847635938773\\
-1.66101694829168	-1.94308690941956\\
-1.58332159323096	-1.61052738801655\\
-1.49646214546838	-1.30956050497635\\
-1.40088940273504	-1.03935446871059\\
-1.29754679071582	-0.79932300422223\\
-1.18788065408612	-0.588939070633224\\
-1.07378704207409	-0.40754384716353\\
-0.957491934186159	-0.25418022463716\\
-0.841377436144445	-0.127480535452816\\
-0.727781051581011	-0.0256296900718262\\
-0.618803143980879	0.0535908927425285\\
-0.516155876887326	0.112690934805796\\
-0.421076127609127	0.154316459798291\\
-0.334309613016835	0.181083718571027\\
-0.256159385122041	0.195430014468024\\
-0.186582972871076	0.199501129503228\\
-0.125319776780633	0.1950856641588\\
-0.0720321818275954	0.183599812184437\\
-0.0264471189433513	0.166123539033287\\
0.0115138858572655	0.14349044333984\\
0.0416317047016969	0.116436602680749\\
0.0633396979726199	0.0858144446819055\\
0.0757707775270669	0.0528697100663606\\
0.0779665840673004	0.0195533987822484\\
0.0693255196460348	-0.01121513852665\\
0.0503384954647456	-0.035509288671893\\
0.023542327853769	-0.0488646462016346\\
-0.00564908495411017	-0.0475870807082389\\
-0.0289980207302727	-0.0310270341951093\\
-0.0375009398120094	-0.00422277057651256\\
-0.0262105680881234	0.0210446259982145\\
-0.000458416454011624	0.0299972109311639\\
0.0220305687384473	0.0153004698968648\\
0.0212717531248203	-0.0111486274598717\\
-0.00266915520074712	-0.0213638661605371\\
-0.0192272962179607	-0.0018995913714139\\
-0.00295346629256246	0.0171001282111415\\
0.0155506076928332	0.00120343247690422\\
-0.00279769210685859	-0.0137449698804074\\
-0.00979019011392662	0.00796449234812817\\
0.0111772025977876	0.002029395756989\\
-0.00718785638829007	-0.00727710612811552\\
0.00396525687005412	0.00831504420068649\\
-0.00337278489386305	-0.00758224361051925\\
0.00500460651994915	0.00555500985494004\\
-0.006667438653148	-0.000969709016226893\\
0.00344256110731107	-0.00500189081050571\\
0.00452866760556426	0.00307295217230577\\
-0.000364776706150085	0.00491959941082933\\
-0.00232794885468293	0.00378886160256293\\
-0.00153884200575176	0.00370168321663358\\
0.00162366986563453	0.00322875773448274\\
0.00280362903811168	-0.00166005201038184\\
-0.00292195576271548	-0.000302344642883905\\
0.00258591346857855	-0.000572355951024326\\
7.41529607694962e-05	0.00238677491349946\\
-0.00112084592277855	0.00183826439956612\\
0.000350049865968089	0.00190943076249535\\
0.00123277270223861	-0.00124254572286298\\
};
\draw (axis cs:-1,.4)[->,>=stealth] arc (90:45:.4) {};
\node at (axis cs: -.7,.53) {\color{black}{\footnotesize{} $\angle{h\jw}$}};
\draw[dashed] (axis cs:-2,-1.27)  -- (axis cs:0.27,1);
\draw[fill=black] (axis cs: -1.52400729267448,	-1.39839417353359) circle (.25ex);
\node at (axis cs: -.35,	-1.39839417353359) {\color{black}{\footnotesize{} $\gamma_i\lft{l}{G_i\jw}{c_i\jw}/j\omega$}};

\end{axis}
\end{tikzpicture}%
    \caption{The black curve shows the Nyquist diagram of $\gamma_i\lft{l}{G_i\s}{c_i\s}/s$ for a particular transfer function. \Cref{eq:basic} is equivalent to checking that each point on this diagram lies in a half-plane  that passes through -1 with angle $\angle{}h\jw$. The margin by which the Nyquist diagram lies within the half-plane is also directly related to the measure of robustness. In particular if the dashed line cuts the real axis at the point $-\delta{}$, then $\gamma_i^*-\gamma_i=\gamma_i\funof{1/\delta-1}$.}
    \label{fig:PHDef}
\end{figure}
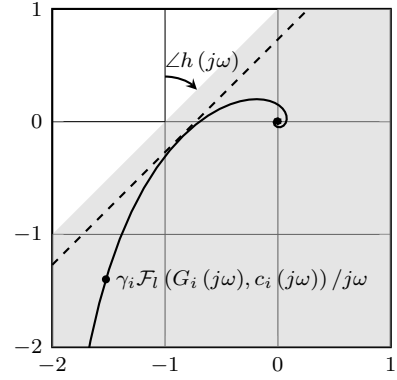
This suggests a simple frequency gridding approach for solving \Cref{prob:analysis}. In particular it shows that \Cref{prob:analysis} is equivalent to
\[
\begin{aligned}
\text{maximize}\;&\gamma\\
\text{subject to}\;&\text{Re}\funof{\!h\!\jw\!\funof{\!1+\frac{\gamma{}\lft{l}{G_i\jw}{c_i\jw}}{j\omega}}\!}\!\geq{}\epsilon,\forall\omega.
\end{aligned}
\]
This optimisation problem is easily tackled with a host of numerical methods. Perhaps more importantly the frequency domain characterization shows that the choice of $h\s$ has a graphical interpretation. To understand this, observe that for a fixed $\omega$, finding an $\epsilon>0$ such that the constraint in the above is satisfied is equivalent to checking whether
\begin{equation}\label{eq:rotplane}
\text{Re}\funof{e^{j\angle{}h\jw}\funof{1+z}}>0,
\end{equation}
where $z={\gamma{}\lft{l}{G_i\jw}{c\jw}}/{j\omega}$. This corresponds to checking whether the point $z\in\C$ lies in a half-plane which cuts through the point $-1$, and has slope $\angle{}h\jw$. This is illustrated in \Cref{fig:PHDef}. The significance of this observation is that it shows that graphical frequency domain tools, robustness measures, and intuition can be used to design both $h\s$ and the controllers $c_i\s$. This will be discussed further in \Cref{sec:res3}. It also connects \Cref{thm:main} to the results from \cite{LV06,PV12}. 

\subsubsection{State-space methods}

If we restrict ourselves to the space of real rational transfer functions, state-space techniques can also be employed. The following simple extension of the \gls{kyp} lemma is the required result. 
It shows that if we have a state-space realisation of the component model and $h$, we can solve \Cref{prob:analysis} by checking an \gls{lmi}. This proof is given in \Cref{app:3}.

\begin{lemma}\label{lem:kyp}
Let $p,h\in\Rat$, $\gamma>0$, and suppose that $p\s,\frac{h\s}{s}$ have minimal realisations
\[
p\s=\sqfunof{\begin{array}{c|c}
    A_1 & B_1 \\\hline
    C_1 & D_1
\end{array}},\;\frac{h\s}{s}=\sqfunof{\begin{array}{c|c}
    A_2 & B_2 \\\hline
    C_2 & 0
\end{array}}.
\]
The following are equivalent:
\begin{enumerate}[(i)]
    \item $\gamma{}p\in\mathcal{P}_h$.
    \item There exists an $X\succ{}0$ such that
    \[
        \begin{bmatrix}
        A^TX+XA&C^T-XB\\C-B^TX&-\funof{D+D^T}
    \end{bmatrix}\prec{}0,
    \]
    where
    \[
    A=\begin{bmatrix}
    A_1&B_1C_2\\0&A_2
    \end{bmatrix},\;B=\begin{bmatrix}
    0\\B_2
    \end{bmatrix},
    \]
    and\qquad
    $
    C=\begin{bmatrix}
    \gamma{}C_1&\gamma{}D_1C_2+C_2A_2
    \end{bmatrix},\;
    D=C_2B_2.
    $
\end{enumerate}
\end{lemma}

Observe in particular that the \gls{lmi} in \Cref{lem:kyp} is affine in $\gamma$. This means that we may address \Cref{prob:analysis} by solving an optimisation problem of the form
\[
\begin{aligned}
\text{maximize}\quad&\gamma\\
\text{subject to}\quad&
        \begin{bmatrix}
        A^TX+XA&C^T-XB\\C-B^TX&-\funof{D+D^T}
    \end{bmatrix}\prec{}0\\
    &X\succ{}0,
\end{aligned}
\]
where $A,B,C,D,\gamma{}$ are as in \Cref{lem:kyp}(ii).

\subsubsection{Circuit theory methods}\label{sec:circtheory} The \gls{pr} functions have also been extensively studied in the context of classical circuit theory. One consequence of this was the development of algebraic tests for positive realness that can be applied to simple functions. For example, excluding the degenerate case $b_0=b_1=b_2=0$, the function
\[
\frac{a_2s^2+a_1s+a_0}{b_2s^2+b_1s+b_0}\in\PR
\]
if and only if all the coefficients are non-negative, and
\begin{equation}\label{eq:biquad}
\begin{aligned}
\funof{\sqrt{a_2b_0}-\sqrt{a_0b_2}}^2\leq{}a_1b_1.
\end{aligned}
\end{equation}
For this result, historical context, and results for other rational functions, see \cite{CS09}. Such tests give a convenient method for solving \Cref{prob:analysis} when $\lft{l}{G_i}{c_i}$ is given by a simple parametrised model. We will illustrate this in \Cref{sec:ex1}.

\subsection{A Scale-Free Design Method}\label{sec:res3}

The true strength of \Cref{thm:main} is that it can be used as a basis for decentralised design with a-priori guarantees that hold for all operating points and network configurations. In this section we will discuss both how to design the function $h\s$, and the local controllers $c_i\s$.

\subsubsection{Designing $h\s$}

The objective here is \textbf{not to design the perfect $h\s$}, rather to get a sensible starting point for designing the decentralised controllers. In \Cref{sec:fr} we saw that testing \cref{eq:basic} with respect to any given $h\s$ is equivalent to checking that the frequency responses of $\gamma_i\lft{l}{G_i\s}{c_i\s}/s$ lie in a frequency dependent half-plane. Therefore if we know roughly how these responses will look, by for example plotting their Nyquist diagrams for some nominal parameter values, we can use this graphical intuition to design a suitable function $h\s$. As illustrated in \Cref{ex:2}, this is extremely easy to do with respect to a fixed half-plane, since a half-plane can be identified directly from the Nyquist diagrams. A function that will certify \cref{eq:basic} for any set of models with Nyquist diagrams in this half-plane is then guaranteed to exist  by the following simple extension of the off-axis circle criterion \cite{CN68}, which is proved in \Cref{app:offaxis}. 

\begin{lemma}\label{lem:offaxiscirc}
Let $p_1,\ldots{},p_n\in\discalg$ and assume that $p_i\funof{0}>0$. If there exists a $\theta\in[0,\pi/2)$ such that for all $i$ 
\[
\text{Re}\funof{e^{j\theta}\funof{1+p_i\jw/j\omega}}>0,\,\forall{}\omega>0,
\]
then there exists an $h\in\PR\cap\discalg$ such that $p_1,\ldots{},p_n\in\mathcal{P}_h$.
\end{lemma}

Even if a fixed half-plane cannot be used, this process can be used to identify frequency ranges where different slopes are suitable. An $h\s$ to match these slopes in the these frequency ranges can then be obtained using a lead-lag design. Alternatively other graphical or computational methods for multiplier design can be used, for example Popov plots. For further discussions about the design of half-planes from the perspective of robustness and performance, see \cite{Pat15}.

\subsubsection{Synthesis of Controllers}

Consider the synthesis counterpart to \Cref{prob:analysis}.
\begin{problem}\label{prob:synthesis}
Given $G_i,h$,
\[
\begin{aligned}
\text{maximize}\quad&\gamma\\
\text{subject to}\quad&\gamma{}\lft{l}{G_i}{c_i}\in\mathcal{P}_h\\
&c_i\in\Rat_{c_i}
\end{aligned}
\]
where $\Rat_{c_i}\subseteq{}\Rat$ denotes the set of possible designs for $c_i$.
\end{problem}

Solving the above maximizes the robustness margin introduced in \Cref{sec:an}. In the power system context, simple controllers are typically desired. In this case the most effective way to solve \Cref{prob:synthesis} is probably to solve the analysis problem in \Cref{prob:analysis} for a range of controller gains, and then select those that maximize $\gamma$. This will be illustrated for \gls{agc} design in \Cref{ex:3}. Alternatively lead-lag design with respect to diagrams such as \Cref{fig:PHDef} offers another simple alternative. Formal synthesis methods can also be used. In fact, when $\Rat_{c_i}=\Rat$ and $G_i\in\Rat$, \Cref{prob:synthesis} can be solved using the \Hfty{} based tools of \cite{SKS94}. 

\begin{theorem}[\cite{SKS94}]\label{thm:hfty}
Let
\[
M=\sqfunof{
\begin{array}{c|cc}
     A&B_1&B_2  \\\hline
     C_1&D_{11}&D_{12}\\
     C_2&D_{21}&0 
\end{array}
},
\]
and assume that $\funof{A,B_2}$ is stabilizable and that $\funof{C_2,A}$ is detectable. Then there exists a strictly proper controller $c\s$ such that $\lft{l}{M}{c}\in\ESPR$ if and only if there exist matrices $X_1,X_2,Y_1,Y_2$ such that
\[
\begin{aligned}
\begin{bmatrix}
AX_1+B_2X_2&0\\
C_1X_1+D_{12}X_2-B_1^T&-D_{11}
\end{bmatrix}+\funof{\star}^T&\prec{}0,\\
\begin{bmatrix}
Y_1A+Y_2C_2&Y_1B_1+Y_2D_{21}-C_1^T\\
0&-D_{11}
\end{bmatrix}+\funof{\star}^T&\prec{}0,\\
\begin{bmatrix}
-X_1&I\\I&-Y_1
\end{bmatrix}&\prec{}0,
\end{aligned}
\]
where $\funof{\star}^T$ denotes the transpose of the matrix on its left.
\end{theorem}

In \cite{SKS94} they also give an explicit realisation of a controller that renders $\lft{l}{M}{c}\in\ESPR$, though due to space limitations we omit this. \Cref{thm:hfty} allows \Cref{prob:synthesis} to be solved as follows. By computing a minimal realisation $M_\gamma{}$ of the transfer function
\[
\begin{bmatrix}
\frac{\gamma{}h\s}{s}&0\\0&I_n
\end{bmatrix}G_i\s+\begin{bmatrix}
h\s&0\\0&0
\end{bmatrix},
\]
and checking the \glspl{lmi} in \Cref{thm:hfty}, the optimal solution to \Cref{prob:synthesis} can be computed to arbitrary precision using a bisection over $\gamma$. Synthesis with further performance and robustness guarantees is also possible by adding more constraints to \Cref{prob:synthesis}. Again, see \cite{Jon01} for an introduction.

\subsection{Do There Exist Better Scale-Free Design Criteria?}\label{sec:d}

\Cref{thm:main} does not offer the only way to conduct scale-free design. For example, passivity theory shows that if for all $i$
\[
\gamma_i\lft{l}{G_i}{c_i}\in\ESPR,
\]
then the power system model is stable\footnote{This is because $\frac{1}{s}L$ is passive for all $L\in\mathcal{L}$, and the negative feedback interconnection of a passive and strictly passive system is stable (e.g. \cite{BL+06}).}. This condition could also be used to conduct decentralised design, and gives the same types of guarantees as \cref{eq:basic}. In this section we will both show that this passivity based condition is a special case of \cref{eq:basic}, and also that in some sense the criteria from \Cref{thm:main} are the best possible. The following demonstrates the first claim, and is proved in \Cref{app:7}.

\begin{lemma}\label{lem:passive}
If $p_1,\dots{},p_n\in\ESPR$, then there exists an $h\in\PR\cap\discalg{}$ such that $p_1,\dots{},p_n\in\mathcal{P}_h$.
\end{lemma}

The converse of \Cref{lem:passive} is not true. Indeed the models considered in \Cref{ex:2,ex:3} are not passive, but do satisfy \cref{eq:basic} for wide ranges of parameter values. In order to investigate whether there are better decentralised stability criteria than \cref{eq:basic}, suppose that for some frequency
\begin{equation}\label{eq:violate}
\text{Re}\funof{h\jw\funof{1+\frac{\gamma_1{}\lft{l}{G_1\jw}{c_1\jw}}{j\omega}}}<0.
\end{equation}
That is \cref{eq:basic} does not hold for the first bus, but perhaps only by an $\epsilon{}$ amount (compare \cref{eq:violate} with the conditions in \Cref{sec:fr}). The idea is that if a better decentralised condition existed, it would have to allow for \cref{eq:violate} to hold. The following theorem shows that for a broad class of functions $h\s$ (which includes all the multipliers used in the examples) this is not possible, since if \cref{eq:violate} holds then there exist $\gamma_2\lft{l}{G_2}{c_2},\dots{},\gamma_n\lft{l}{G_n}{c_n}\in\mathcal{P}_h$ and an $L_B$ meeting \Cref{ass:1} such that the power system model is unstable. This means that we cannot even relax the decentralised requirement for a single component by an $\epsilon{}$ amount and still obtain a-priori stability guarantees in a decentralised manner. 
\begin{theorem}\label{thm:converse}
Let $p_1\in\Hfty$, 
\[
h\s=\frac{s}{s+T}g\s\in\PR\cap\discalg,
\]
where $T>0$ and $g,g^{-1}\in\discalg$, and assume that \cref{eq:violate} holds for some $\omega>0$. Then given any $n\geq{}2$ there exist $p_2,\ldots{},p_n\in\mathcal{P}_h$ and an $L\in\mathcal{L}$ such that \cref{eq:seceq} is unstable.
\end{theorem}
\begin{proof}
The interconnection in \cref{eq:seceq} is stable only if
\[
M\jw=\funof{I+L\diag\funof{p_1\jw,\ldots{},p_n\jw}/j\omega}
\]
is invertible. Now suppose that $p_2\!\s\!=\!\ldots{}\!=\!p_n\!\s\!=\!p\!\s$, and
\[
L=\begin{bmatrix}
1/\sqrt{2}\\-\sqrt{\frac{1}{2\funof{n-1}}}\mathbf{1_{n-1}}
\end{bmatrix}\begin{bmatrix}
1/\sqrt{2}\\-\sqrt{\frac{1}{2\funof{n-1}}}\mathbf{1_{n-1}}
\end{bmatrix}^T.
\]
Under these conditions $L\in\mathcal{L}$ and
\[
\det{}M\jw=1+\tfrac{1}{2}p_1\jw/j\omega+\tfrac{1}{2}p\jw/j\omega.
\]
Letting $x=p_1\jw/j\omega$, we see that if $p\jw/j\omega\equiv{}-x-2$ then $\det{}M\jw=0$, and therefore $M\jw$ is not invertible. Therefore all we need to do is find a $p\in\mathcal{P}_h$ such that $p\jw/j\omega=-x-2$. Equivalently we can find a $q\in\ESPR$ such $q\jw=h\jw\funof{-1-x}$ and $q\funof{\infty}=h\funof{\infty}$, and then set
\[
p\s=\funof{s+T}\funof{q\s-h\s}g\s^{-1}.
\]
Provided $\text{Re}\funof{h\jw\funof{-1-x}}>0$, such a $q$ can always be found using well known interpolation results (for example \cite[Lemma 1.14]{Vin00}). Observing that by assumption
\[
\text{Re}\funof{h\jw\!\funof{-1-x}}\!=\!-\text{Re}\funof{h\jw\!\funof{1+\tfrac{p_1\jw}{j\omega}}}>0
\]
completes the proof.
\end{proof}

\section{Examples}\label{sec:examples}

The three examples in this section show that our conditions can be used to: (i) demonstrate stability of existing power system models; (ii) give delay robustness guarantees for the swing dynamics with delayed droop control; and (iii) analyse the robust stability of automatic generation control (AGC) and design novel AGC controllers.

\subsection{Stability of the Swing Equations}\label{sec:ex1}

In this example we will show that our criteria can be used to verify stability of the swing equations when there is no control. It is or course no great surprise that this model is stable, and many other tools can be used to prove this. It is nevertheless reassuring that our conditions can easily cover this case.

If we have a swing equation model with no control, then for all $i$, $c_i=0$, and consequently
\begin{equation}\label{eq:ex1show}
\lft{l}{G_i}{c_i}=\frac{1}{m_is+d_i},
\end{equation}
Therefore in this case, \cref{eq:basic} simplifies to
\begin{equation}\label{eq:ex1}
\frac{\gamma_i}{m_is+d_i}\in\mathcal{P}_h.
\end{equation}
The following corollary shows that there exists an $h$ such that the above holds for arbitrarily large $\gamma_i$ given any $m_i\geq{}0$ and $d_i>0$. Therefore the swing equation model is stable by \Cref{thm:main} for any possible parameter values, operating point and interconnection configuration. The proof uses the tools from circuit theory discussed in \Cref{sec:circtheory}, illustrating their strength when simple parametrised models are considered.

\begin{corollary}\label{ex:1}
Let $p_1\s=\gamma_1/\funof{m_1s+d_1},\ldots{},p_n\s=\gamma_n/\funof{m_ns+d_n}$. If for all $i$
\[
m_i\geq{}0,\, d_i>0\,\text{and}\,\gamma_i>0,
\]
then there exists an $h\in\PR\cap\discalg{}$ such that $p_1,\ldots{},p_n\in\mathcal{P}_h$.
\end{corollary}
\begin{proof}
Let $h\s=\frac{s}{Ts+1}$.  It is sufficient to show that for all $i$ there exists an $\epsilon>0$ such that
\[
\frac{s}{s+T}\funof{1+\frac{\gamma_i}{s\funof{m_is+d_i}}}-\epsilon\in\PR.
\]
Multiplying out the above shows that it is equivalent to
\[
\frac{\funof{1-\epsilon}m_is^2+\funof{d_i-d_i\epsilon{}-T\epsilon{}m_i}s+\gamma_i-Td_i\epsilon{}}{m_is^2+\funof{d_i+Tm_i}s+Td_i}\in\PR.
\]
We can show that the above holds by applying \cref{eq:biquad}. Note however that $\funof{\sqrt{a_2b_0}-\sqrt{a_0b_2}}^2\leq{}\max\cfunof{a_2b_0,a_0b_2}$, and that if $T$ is sufficiently large and $\epsilon$ sufficiently small, then for all $i$
\[
\funof{1-\epsilon}m_iTd_i\geq{}m_i\funof{\gamma_i-Td_i\epsilon{}}.
\]
Therefore it is sufficient to show that $\funof{1-\epsilon}m_iTd_i\leq{}\funof{d_i-d_i\epsilon{}-T\epsilon{}m_i}\funof{d_i+Tm_i}$. Multiplying out this expression yields
\[
d_i^2-\funof{d_i\funof{d_i+Tm_i}+T^2m_i^2}\epsilon\geq{}0.
\]
We can always pick $\epsilon$ small enough so that the above holds for all $i$, which completes the proof.
\end{proof}

\subsection{Stability of Droop Control Subject to Delay}\label{ex:2}

\begin{figure}
    \centering
    \input{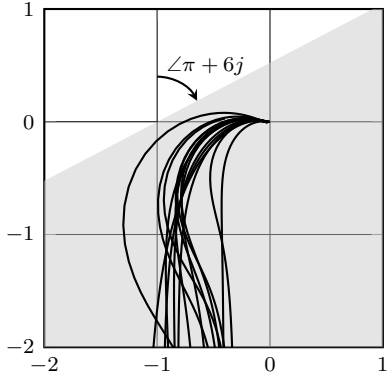}
    \caption{The curves show the Nyquist diagrams of $\gamma_i\lft{l}{G_i}{c_i}/s$ for a range of parameter values meeting \cref{eq:scfrdroop,eq:scfrdroop1}. By \Cref{lem:delays}, all these curves lie within the same half-plane, and the effect of increasing the delay is to push the curves closer to its boundary.}
    \label{fig:delaydroop}
\end{figure}

In this example we will use our criteria to verify stability of the swing equations when there is droop control subject to delays. In order to get simpler criteria we will neglect governor and turbine dynamics (these can easily be included, and will be in the next example). This model is described by
\begin{equation}\label{eq:delay-swing}
\begin{aligned}
G_i=\frac{1}{m_is+d_i}\begin{bmatrix}
1&1\\1&1
\end{bmatrix},\;c_i=-\frac{1}{r_i}e^{-s\tau_i},\\\Longrightarrow{}\lft{l}{G_i}{c_i}=\frac{1}{m_is+d_i+\frac{1}{r_i}e^{-s\tau_i}}.
\end{aligned}
\end{equation}
In the above $r_i>0$ is the droop constant, and $\tau_i\geq{}0$ a measurement delay.

In the following we will show that if for all $i$
\begin{equation}\label{eq:scfrdroop}
r_i\leq{}\sqrt{2/\gamma_im_i},
\end{equation}
then stability of the power system model is guaranteed by \Cref{thm:main} for any values of the delays that satisfy
\begin{equation}\label{eq:scfrdroop1}
0\leq{}\tau_i<\pi{}m_ir_i/4,
\end{equation}
and for any non-negative values of the natural damping constants $d_i$ (which are typically unknown). This perfectly illustrates the strength of our approach for conducting design in the network setting. By using \Cref{thm:main}, the task of synthesizing decentralised controllers to guarantee robust stability to delays in a large uncertain system --a daunting task-- has been simplified to picking a set of constant gains that satisfy a simple inequality. Such constants always exist, and the resulting controllers are simple to implement. Furthermore the design comes with a-priori guarantees about robustness to delays and levels of natural damping, that hold entirely independently of operating point and interconnection configuration.

To derive this result we will use the approach outlined in \Cref{sec:res3}. As suggested there, in order to choose a suitable $h\s$, we plot the Nyquist diagrams of $\lft{l}{G_i\s}{c_i\s}/s$ for a range of parameter values. This is shown in \Cref{fig:delaydroop}. This not only shows that passivity tools cannot be used, even for arbitrarily small values of the delay, but also that the Nyquist diagrams lie within the same half-plane for wide ranges of parameter values. This suggests that we can use \Cref{lem:offaxiscirc} to verify the decentralised stability requirement in \cref{eq:basic}. In fact this requirement can be turned into parameter dependent inequalities, as shown in \Cref{lem:delays} below. For ease of presentation we only give the result for the special choice of half-plane that leads to \cref{eq:scfrdroop,eq:scfrdroop1}. For generalizations of these inequalities and the proof, see \Cref{app:6}.

\begin{lemma}\label{lem:delays}
Let $m\geq{}0,r>0$ and $\gamma{}>0$. If
\[
r\leq{}\sqrt{2/{\gamma}m},
\]
then for all $0\leq{}\tau<\pi{}mr/4,d\geq{}0$ and $\omega>0$,
\[
\text{Re}\funof{\funof{\pi+6j}\funof{1+\frac{\gamma{}}{j\omega\funof{mj\omega+d+\frac{1}{r}e^{-j\omega{}\tau}}}}}>0.
\]
\end{lemma}

\subsection{Stability of Automatic Generation Control (AGC)}\label{ex:3}

\begin{figure}
\centering
\begin{tikzpicture}[scale=.67,font=\footnotesize]
  \node[fill=white,draw, circle, inner sep=0, radius=0.2] (p1) at (0,.5) {$+$};
  \draw[fill=white] (.7,0) rectangle node[midway] {$\frac{k_i}{s}$} (1.5,1);
  \node[fill=white,draw, circle, inner sep=0, radius=0.2] (p2) at (2.2,.5) {$+$};
  \draw[fill=white] (2.9,0) rectangle node[midway] {$\frac{1}{1+sT_{g,i}}$} node[below, yshift=-8]{Governor} (4.5,1);
  \draw[fill=white] (5,0) rectangle node[midway] {$\frac{1}{1+sT_{t,i}}$} node[below, yshift=-8]{Turbine} (6.6,1);
  \node[draw, circle, inner sep=0, radius=0.2] (p3) at (7.4,.5) {$+$};
  \draw (8.1,0) rectangle node[midway] {$\frac{1}{m_is+d_i}$} node[below, yshift=-8]{Generator} (9.7,1);
  \draw[fill=white] (-.4,1.35) rectangle node[midway] { $\text{-}\beta_i$} (.4,2.35);
  \draw[fill=white](1.8,1.35) rectangle node[midway] {$\text{-}\frac{1}{r_i}$} (2.6,2.35);
  \draw[->] (-1.7,0.5) -- node[above, pos=.25] {$P_{d,i}\!-\!P_{N,i}$} (p1);
  \draw[->] (p1) -- (.7,.5);
  \draw[->] (1.5,.5) -- (p2);
  \draw[->] (p2) -- (2.9,.5);
  \draw[->] (4.5,.5) -- (5,.5);
  \draw[->] (6.6,.5) -- (p3);
  \draw[->] (p3) -- (8.1,.5);
  \draw[->] (9.7,.5) -- node[above, pos=.75] {$\dot{\theta_i}$} (10.5,.5);
  \draw[->] (9.85,.5) -- (9.85,3) -- (0,3) -- (0,2.35);
  \draw[->] (0,1.35) -- (p1);
  \draw[->] (2.2,1.35) -- (p2);
  \draw[->] (2.2,3) -- (2.2,2.35);
  \draw[->] (-1.2,0.5) -- (-1.2,-1.5) -- (7.4,-1.5) -- (p3);
\end{tikzpicture}
{\footnotesize
\begin{tabular}{p{.07\columnwidth} p{.07\columnwidth} p{.07\columnwidth} p{.07\columnwidth} p{.07\columnwidth} p{.07\columnwidth} p{.07\columnwidth}}
\\ [0.2ex]
\toprule
${}m{}$ & ${}d$ & ${}T_g$ &${}T_t$&${}r$&${}\beta$&${}k$\\ [0.2ex]
\midrule
0.16&0.02&0.08&0.40&3.00&0.33&0.30\\
0.20&0.02&0.06&0.44&2.73&0.40&0.20\\
0.12&0.02&0.07&0.30&2.82&0.38&0.40\\
\bottomrule
\end{tabular}}
\vspace{.3cm}
\caption{\label{fig:agcblock}Typical \gls{agc} controller architecture and parameters \cite{Bev14}.}
\end{figure}

\Gls{agc} is an extension of droop control. The primary objective of \gls{agc} is to regulate system frequency to the specified nominal value (50/60 Hz), while maintaining the flow of power between buses at their scheduled values. A typical controller architecture is shown in \Cref{fig:agcblock} \cite{Bev14}. From the control perspective, the synthesis task is to design the parameters $\beta_i,k_i$. It is common to select $\beta_i\approx{}1/r_i+d_i$,
with $k_i$ selected based on simulation studies to act on the time scale of 1-10 minutes (see e.g. \cite[\S{11.1.5}]{kundur_power_1994}), and it has been observed that when `large' $\beta_i$'s are chosen, stability issues can arise.

Within our framework, the generalised plant is
\[
G_i\s={\begin{bmatrix}
\frac{1}{m_is+d_i}&\frac{1}{\funof{m_is+d_i}\funof{1+sT_{g,i}}\funof{1+sT_{t,i}}}\\
\frac{1}{m_is+d_i}&\frac{1}{\funof{m_is+d_i}\funof{1+sT_{g,i}}\funof{1+sT_{t,i}}}\\
1&0
\end{bmatrix}},
\]
and the standard \gls{agc} controller is
\[
c_i\s=\begin{bmatrix}
-\frac{1}{r_i}&0
\end{bmatrix}+\frac{k_i}{s}\begin{bmatrix}
-\beta_i&1
\end{bmatrix}.
\]
To formally address the design of the \gls{agc} controller, we solved the analysis problem in \Cref{prob:analysis} for a range of values of the control parameters. For the first set of generator parameters this is shown in \Cref{fig:agctradeoff}. From this figure we see that the nominal design, which is marked by a cross, is a reasonable choice, though the robustness margin could be further improved by reducing $\beta_i$ or increasing $k_i$. We also see that increasing $\beta_i$ will reduce the optimal $\gamma$, justifying the observation that `large' $\beta_i$'s can cause stability problems.

\begin{figure}
    \centering
    \input{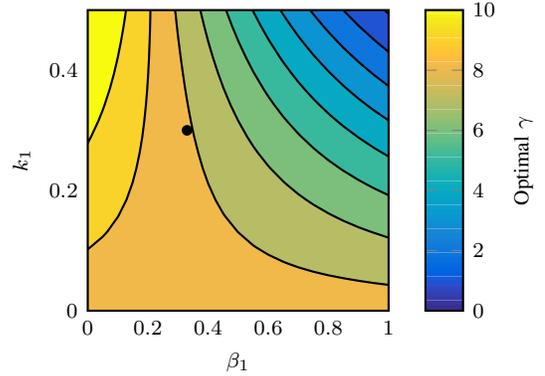}
    \caption{Solution to \Cref{prob:analysis} for the \gls{agc} model with the first set of parameters in \Cref{fig:agcblock} over a range of values for $\beta_1,k_1$. The nominal design is marked by the black dot.}
    \label{fig:agctradeoff}
\end{figure}

We can also design \gls{agc} controllers by solving the synthesis problem in \Cref{prob:synthesis} using \Hfty{} methods. Given the need for simple controllers, the value here is more in finding out what levels of robustness are possible, rather than in the controllers themselves. To this end we fixed the controller parameters $r_i,\beta_i$ to their values from \Cref{fig:agcblock}. Selecting the best possible $k_i\in\R$ gives an optimal solution of around $11$. However, by replacing the constant $k_i$ with a transfer function $k_i\in\Rat$, and solving the synthesis problem using the $\Hfty{}$ method from \Cref{sec:res3} yields an optimal solution of around $10^4$. This shows that the use of dynamic control has the potential to greatly increase the robustness margin. It is interesting to think how this can be exploited in the design of inverters, where the use of more complex controllers is a more realistic prospect.

\section{Conclusions}

A decentralised analysis and design framework for frequency control in power systems has been presented. Our framework allows for the design of decentralised controllers using only local models, and provides strong a-priori robust stability guarantees that hold independently of operating point, even as components are added to and removed from the grid. Furthermore our conditions can be applied even when the network consists of complex heterogeneous components, and can be checked using standard frequency response, state-space, and circuit theoretic tools. We illustrate the suitability of the framework for power systems by: (i) showing that the robustness of existing schemes can be analysed and further improved using the newly developed tools; and (ii) providing novel delay robustness criteria for the classical swing equations.


\bibliographystyle{IEEEtran}
\bibliography{Refs.bib,Refs-em.bib}
%

%


 \vspace{-6ex}
\begin{IEEEbiographynophoto}{Richard Pates}
received the M.Eng degree in 2009, and the Ph.D. degree in 2014, both from the University of Cambridge. He is currently a Researcher at Lund University. His research interests include scale-free methods for control system design, stability and control of electrical power systems, and fundamental performance limitations in large-scale systems.
\end{IEEEbiographynophoto}
 \vspace{-7ex}
\begin{IEEEbiographynophoto}{Enrique Mallada}
 is an assistant professor of electrical and computer engineering at Johns Hopkins University. Before joining Hopkins in 2016, he was a post-doctoral fellow at the Center for the Mathematics of Information at the California Institute of Technology from 2014 to 2016. He received his ingeniero en telecomunicaciones degree from Universidad ORT, Uruguay, in 2005 and his Ph.D. degree in electrical and computer engineering with a minor in applied mathematics from Cornell University in 2014. 
  \vspace{-3ex}
\end{IEEEbiographynophoto}






\appendix

\subsection{Proof of \Cref{lem:rescale}}\label{app:lem1}

\begin{proof}
\Cref{ass:1} implies that $0\preceq{}L_B$, from which standard arguments (using e.g. Gershgorin discs) show that
\[
0\preceq{}\begin{bmatrix}
\Gamma_1&0\\0&\Gamma_2
\end{bmatrix}^{-\frac{1}{2}}\begin{bmatrix}
L_{B,11}&L_{B,11}\\
L_{B,21}&L_{B,22}
\end{bmatrix}\begin{bmatrix}
\Gamma_1&0\\0&\Gamma_2
\end{bmatrix}^{-\frac{1}{2}}\preceq{}I.
\]
The result then follows immediately from \cite[Theorem 5]{Smi92}.
\end{proof}

\subsection{Proof of \Cref{lem:marg}}\label{app:1}

\begin{proof}
Since $p\in\mathcal{P}_h$, there exists an $\epsilon$ such that $h\s\funof{1+\frac{p}{s}}-\epsilon\in\PR$. Therefore
\[
\frac{1-\gamma}{\gamma}h\s+h\s\funof{1+\frac{p\s}{s}}-\epsilon\in\PR.
\]
This implies that $h\s\funof{1+\gamma{}\frac{p\s}{s}}-\gamma\epsilon\in\PR$. Consequently $\gamma{}p\s\in\mathcal{P}_h$ for all $0<\gamma\leq{}1$ as required.
\end{proof}

\subsection{Proof of \Cref{lem:fr}}\label{app:2}

\begin{proof}
Denote $\phi\s=\frac{1-s}{1+s}$, and let $z=\phi\s$ and $G\funof{z}=g\funof{\phi^{-1}\funof{z}}$. Since $\phi$ maps the open right half plane to the open unit circle,
\[
\sup_{\text{Re}\funof{s}>0}\text{Re}\funof{g\s}=\sup_{\abs{z}<1}\text{Re}\funof{G\funof{z}}.
\]
Since $g\s\in\discalg$, $G\funof{z}$ is analytic in the open unit circle, and continuous on the unit circle \cite{Par97}. Therefore by the maximum modulus principle
\[
\sup_{\abs{z}<1}\text{Re}\funof{G\funof{z}}=\!\!\max_{t\in\sqfunof{0,2\pi}}\text{Re}\funof{G\funof{e^{jt}}}=\!\!\!\!\!\max_{\omega\in\R\cup\cfunof{\infty}}\text{Re}\funof{g\s}.
\]
The result is now immediate from \Cref{def:pr}.
\end{proof}

\subsection{Proof of \Cref{lem:kyp}}\label{app:3}

\begin{proof}
By \cite[Lemma 2.3]{SKS94}, if
\[
G\s=\sqfunof{\begin{array}{c|c}
    A & B \\\hline
    C & D
\end{array}},
\]
then the condition $G\in\ESPR$ is equivalent to the existence of an $X\succ{}0$ such that
\[
    \begin{bmatrix}
        A^TX+XA&C^T-XB\\C-B^TX&-\funof{D+D^T}
    \end{bmatrix}\prec{}0.
\]
Therefore we need only show that
\[
h\s\funof{1+\frac{\gamma{}p\s}{s}}=\sqfunof{\begin{array}{c|c}
    A & B \\\hline
    C & D
\end{array}},
\]
where $A,B,C,D$ are given as in (ii). Applying standard formulae for multiplying state-space realisations shows that $h\s\gamma{}p\s/s$ and $sh\s$ have realisations
\[
\begin{aligned}
\sqfunof{\begin{array}{cc|c}
    A_1&B_1C_2 & 0 \\
    0&A_2&B_2\\\hline
    \gamma{}C_1&\gamma{}D_1C_2 & 0
\end{array}}\;\text{and}\;
\sqfunof{\begin{array}{c|c}
    A_2 & B_2 \\\hline
    C_2A_2 & C_2B_2
\end{array}}
\end{aligned}
\]
respectively, from which the result immediately follows.
\end{proof}

\subsection{Proof of \Cref{lem:offaxiscirc}}\label{app:offaxis}

\begin{proof}
Let $g_i=\funof{s/\funof{s+T}}\funof{1+p_i\s/s}$. It is easily shown that $g_i\in\discalg$, and that for $T$ sufficiently large there exists an $\epsilon>0$ such that for all $i$ and $\omega\geq{}0$ 
\[
\text{Re}\funof{-je^{j\funof{\theta+1/T}}g_i\jw}\geq{}\epsilon{}.
\]
Therefore by \cite[Theorem 2]{CN68} there exists an `$RC$' multiplier $h_{RC}$ such that $h_{RC}g_i\in\ESPR$. Consequently if $h\s=h_{RC}\s{}s/\funof{s+T}$, then $p_1,\ldots{},p_n\in\mathcal{P}_h$ as required.
\end{proof}

\subsection{Proof of \Cref{lem:passive}}\label{app:7}

\begin{proof}
Since $p_i\in\ESPR$ there exists an $\epsilon>0$ and a $\gamma>0$ such that for all $i$ and $\omega\in\R\cup\cfunof{\infty}$,
\[
\text{Re}\funof{p_i\jw}\geq{}\epsilon{},\,\abs{\text{Im}\funof{p_i\jw}}\leq{}\gamma.
\]
Let $h\s=s/\funof{s+T}$. By \Cref{lem:fr}, $p_i\in\mathcal{P}_h$ if and only if there exists and $\delta>0$ such that
\[
\text{Re}\funof{j\omega/\funof{j\omega+T}\funof{1+p_i\jw/j\omega}}\geq{}\delta{}.
\]
This is equivalent to
\[
\begin{aligned}
\frac{T\text{Re}\funof{p_i\jw}+\omega\funof{\omega+\text{Im}\funof{p_i\jw}}}{\omega^2+T^2}&\geq{}\delta{}&\Longleftarrow{}\\
\frac{T\epsilon+\omega\funof{\omega-\text{sign}\funof{\omega}\gamma{}}}{\omega^2+T^2}&\geq{}\delta{}.
\end{aligned}
\]
By picking $T$ sufficiently large there will always exist a $\delta>0$ such that the above is satisfied, which completes the proof.
\end{proof}

\subsection{Proof of \Cref{lem:delays}}\label{app:6}

\begin{figure}
    \centering
\begin{tikzpicture}
 
  \def\phi{(0.65*pi)} 
  \def\phia{(0.55*pi)} 
  \def\phib{(0.45*pi)} 
  \def\phic{(0.35*pi)} 
  \def\phid{(0.25*pi)} 
  \def\phie{(0.15*pi)} 
  \def\valtheta{0.495} 
  \def\xintersect{-0.8244} 
 
  \begin{axis}[axis equal image, axis lines=middle, domain=0:50, samples=10000, xmax =
1.5, xmin = -5, ymax = 1, ymin = -2,xtick={-1},ytick={-4},extra y ticks={-1.9099},extra y tick style={yticklabel={$-\frac{6}{\pi}$}}]
    \addplot[mark=none, smooth, samples=250, variable=r,domain=0:180]
      ({-1+2.1558*cos(r)},
       {-1.9099+2.1558*sin(r)});
       \addplot[mark=none, gray, smooth, samples=250, variable=r,domain=0:3]
      ({-r*(r-sin(deg(0.68*r))},
       {r*(0.1+cos(deg(0.68*r)))});
       \addplot[mark=none, gray, smooth, samples=250, variable=r,domain=0:3]
      ({-r*(r-sin(deg(0.71*r))},
       {r*(cos(deg(0.71*r)))});
       \addplot[mark=none, smooth, samples=250, variable=r,domain=0:3]
      ({-r*(r-sin(deg(0.7854*r))},
       {r*cos(deg(0.7854*r))});
    \draw[->,,>=stealth] (axis cs:-1,-1.9099) -- node[left]{$\sqrt{1+\frac{36}{\pi^2}}$} (axis cs:-2,0);
  \end{axis}
 
\end{tikzpicture}
    \caption{Sketch of the geometric argument used in the proof of \Cref{lem:delays}. The critical curve (that just touches the circle) is found by setting $d=0$, and finding $t$ to match the slopes of the circle and the curve at their point of intersection. The effect of increasing $d$ and decreasing $t$ is to shift the curve away from the circle, as shown in grey.}
    \label{fig:sephyp}
\end{figure}
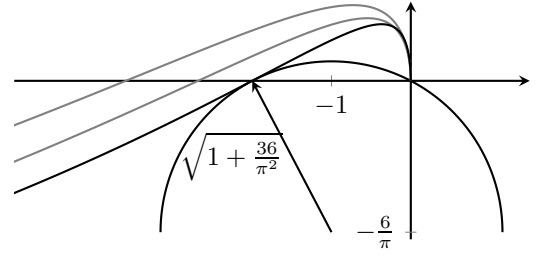

\begin{proof}
First note that by putting $k=1/{m\gamma{}r^2}$, $\tilde{\omega}=mr\omega$ and $t={\tau}/{mr}$,
we obtain the following canonical form:
\[
\frac{1}{j\omega{}\funof{mj\omega{}+d+\frac{1}{r}e^{-j\omega{}\tau}}}=\frac{\nicefrac{1}{k}}{j\tilde{\omega}\funof{j\tilde{\omega}+d/mrk+e^{-tj\tilde{\omega}}}}.
\]
From the conditions of the theorem $k\geq{}\frac{1}{2}$, and therefore it is sufficient to show that
\begin{equation}\label{eq:suffirst}
\text{Re}\funof{\funof{\pi+6j}\funof{\frac{1}{2}+\frac{1}{j\tilde{\omega}\funof{j\tilde{\omega}+d/mrk+e^{-tj\tilde{\omega}}}}}}>0.
\end{equation}
Given $z\neq{}0$ it is simple to show that
\[
\begin{aligned}
\text{Re}\funof{\funof{\pi+6j}\funof{1/2+1/z}}&>0&\Longleftrightarrow{}\\
\funof{\text{Re}\funof{z}+1}^2+\funof{\text{Im}\funof{z}+6/\pi}^2&>\sqrt{1+36/\pi^2}.
\end{aligned}
\]
Therefore if the curve ${j\tilde{\omega}\funof{j\tilde{\omega}+d/mrk+e^{-tj\tilde{\omega}}}}$ lies strictly outside a circle with centre $-1-6j/\pi$ and radius $\sqrt{1+36/\pi^2}$, then the theorem holds. A lengthy but routine geometric argument then shows that this is the case for all $d\geq{}0$ and $\tilde{\omega}>0$ if and only if $t<\pi/4$, from which the result follows. This is illustrated in \Cref{fig:sephyp}.
\end{proof}
The following generalization allows for arbitrary half-planes (\Cref{lem:delays} corresponds to the case $\alpha=\pi/4$). Reducing $\alpha$ allows for stronger delay robustness guarantees at the expense of requiring larger droop constants $r$.
\begin{lemma}
Let $m\geq{}0,r>0,\gamma{}>0$ and $\pi/2>\alpha>0$. If
\[
r\leq{}\sqrt{\frac{\pi\funof{\pi-2\alpha{}}}{4\alpha^2m\gamma}},
\]
then for all $0\leq{}\tau<\alpha{}mr,d\geq{}0$ and $\omega>0$,
\[
\text{Re}\!\funof{\!\!\funof{\pi{}\!+\!\frac{2j\funof{\pi-\alpha}}{\alpha}}\!\!\funof{1\!+\!\frac{\gamma{}}{j\omega\funof{mj\omega+d+\frac{1}{r}e^{-j\omega{}\tau}}}}\!\!}\!\!>\!0.
\]
\end{lemma}

\end{document}